\documentclass[reqno]{amsart}

\usepackage[left=3cm,top=3cm,right=3cm,bottom=3cm]{geometry}

\usepackage{color}
\usepackage[pdfstartpage=1,colorlinks=true,bookmarks=false,pdfstartview={FitH}]{hyperref}

\usepackage{amsmath,amssymb}
\usepackage{mathtools}
\usepackage{verbatim}
\usepackage{comment}
\newcommand{\COMMENT}[1]{}
\DeclarePairedDelimiter\abs{\lvert}{\rvert}
\makeatletter
\let\oldabs\abs
\def\abs{\@ifstar{\oldabs}{\oldabs*}}
\newcommand{\eps}{\varepsilon}

\newcommand{\R}{\mathbb{R}}
\newcommand{\p}{\partial}

\newcommand{\norm}[2][]{\left\|{#2}\right\|_{#1}}

\newcommand{\set}[1]{\left\{#1\right\}}
\newcommand{\absgrad}[1]{\abs{\nabla{#1}}}
\newcommand{\Bn}{{\mathbb{B}^N}}
\newcommand{\absgradB}[1]{\abs{\nabla_{\Bn}{#1}}}

\newcommand{\texton}{\text{ on }}
\newcommand{\textin}{\text{ in }}
\newcommand{\textforall}{\text{ for all }}
\newcommand{\textand}{\text{ and }}
\newcommand{\textsince}{\text{ since }}
\newcommand{\DeltaB}{\Delta_{\Bn}}

\newtheorem{theorem}{Theorem}[section]

\newtheorem{lemma}[theorem]{Lemma}

\theoremstyle{definition}

\newtheorem{remark}{Remark}

\newtheorem{lem}{Lemma}[section]

\newtheorem{thm}{Theorem}[section]
\newtheorem{prop}{Proposition}[section]
\newcommand{\bremark}{\begin{remark} \em}
\newcommand{\eremark}{\end{remark} }

\numberwithin{equation}{section}

\definecolor{g2}{rgb}{0,0.6,0}
\definecolor{r2}{rgb}{0.8,0,0}
\newcommand{\pd}[2]{\frac{\partial#1}{\partial#2}}

\vbadness=\maxdimen

\begin{document}

\title[Hyperbolic Caffarelli--Kohn--Nirenberg inequality]{A variational problem associated to a hyperbolic Caffarelli--Kohn--Nirenberg inequality}

\author[H.~Chan]{Hardy Chan}
\author[L.F.O.~Faria]{Luiz Fernando de Oliveira Faria}
\author[S.~Shakerian]{Shaya Shakerian}
\address{Department of Mathematics, University of British Columbia, Vancouver, B.C., Canada, V6T 1Z2}
\email[H.~Chan]{hardy@math.ubc.ca}
\email[S.~Shakerian]{shaya@math.ubc.ca}

\address{Departamento de Matem\'{a}tica, Universidade Federal de Juiz de Fora}
\email[L.F.O.~Faria]{lfofaria@gmail.com}
\begin{abstract}
We prove a Caffarelli--Kohn--Nirenberg inequality in the hyperbolic space. For a semilinear elliptic equation involving the associated weighted Laplace--Beltrami operator, we establish variationally the existence of positive radial solutions in the subcritical regime. We also show a non-existence result in star-shaped domains when the exponent is supercritical.
\end{abstract}

\maketitle

\tableofcontents

\section{Introduction}

\subsection{The Caffarelli--Kohn--Nirenberg inequality}

In \cite{CKN}, Caffarelli, Kohn and Nirenberg proved the following celebrated interpolation inequality which states that in any dimension
$N\geq 3$, there is a constant $C=C(a,b,N)>0$ such that for all $u\in C_c^{\infty}(\R^N),$ the following holds:

\begin{equation}\label{eq:CKNeucl}
\left(\int_{\R^N}\abs{x}^{-bp}\abs{u}^p\,dx\right)^{\frac2p}
    \leq C \int_{\R^N}\abs{x}^{-2a}\absgrad{u}^2\,dx,
\end{equation}
where
$$-\infty<a<\frac{N-2}{2}, \quad 0\leq{b-a}\leq1 \quad\textand\quad p=\frac{2N}{N-2+2(b-a)}.$$

The best constant in the above inequality is defined as  \begin{equation}\label{eq:Sab}
S(a,b,\R^N):= \inf\limits_{u \in D^{1,2}_a(\R^N) \setminus \{0\}} \frac{\int_{\R^N}\abs{x}^{-2a}\absgrad{u}^2\,dx}{\left(\int_{\R^N}\abs{x}^{-bp}\abs{u}^p\,dx\right)^{\frac2p}},
\end{equation}
where $D^{1,2}_a(\R^N)$ is the completion of $C^\infty_c(\R^N)$ with respect to the norm $\|u\|^2_a=\int_{\R^N}\abs{x}^{-2a}\absgrad{u}^2\,dx$. It was computed explicitly by Aubin \cite{Aubin}, Talenti \cite{Talenti} for the Sobolev inequality when $a=b=0$, by Lieb \cite{Lieb} for the case $a=0$, $0<b<1$, then by Chou--Chu \cite{Chou-Chu} in the full $a$-non-negative region $0\leq{a}<\frac{N-2}{2}$, $a\leq{b}\leq{a+1}$. Minimizers, which are radial, were given explicitly using classical ODE analysis. The picture was completed by Catrina--Wang \cite{CW}, who investigated the parameter region with $a<0$. They observed, among other things, a symmetry breaking phenomenon. Namely, the best constant is attained by a non-radial minimizer.

The existence or non-existence of minimizers 
and their qualitative properties as well as improved versions with remainders have been extensively studied over the last two decades. See, for instance, Abdellaoui--Colorado--Peral \cite{Abdellaoui-Colorado-Peral}, Catrina--Wang \cite{CW}, Dolbeault--Esteban--Loss--Tarantello \cite{DELG}, Felli \cite{Felli}, N\'{a}poli--Drelichman--Dur\'{a}n \cite{Napoli-Drelichman-Duran}, Sano--Takahashi \cite{Sano-Takahashi}, Shen--Chen \cite{Shen-Chen} and references therein.


In this paper, we prove this family of Caffarelli--Kohn--Nirenberg (C--K--N) inequalities \eqref{eq:CKNeucl} 
on the disc model of the Hyperbolic space $\mathbb{B}^N.$

\begin{thm}[A C--K--N inequality on Hyperbolic space]
\label{thm:CKN}
For any $-\infty<a<\frac{N-2}{2}$, $0\leq{b-a}\leq1$ and $p=\frac{2N}{N-2+2(b-a)}$, there exists a constant $C=C(a,b,N)>0$ such that
\begin{equation}\label{ckn}
\left(\int_{\Bn}d^{-bp}\abs{u}^p\,dV\right)^{\frac2p}\leq{C}\int_{\Bn}d^{-2a}\absgradB{u}^2\,dV, \mbox{ for } u\in C_c^{\infty}(\Bn),
\end{equation}
where we denote the hyperbolic distance function from the origin by $d:=d_{\mathbb{B}^N}(0,x)=\log\frac{1+\abs{x}}{1-\abs{x}}$, the hyperbolic gradient by $\nabla_{\Bn}$, and the hyperbolic volume element by $dV$, as introduced in Section \ref{sec:prelim}.
\end{thm}

The proof is given in Section \ref{sec:CKN}. Some remarks are in order.

\begin{remark}
One can deduce from the proof 
that ineuality \eqref{ckn} holds for weights other than $d$, which are radially decreasing and have a quadratic singularity at the origin.
\end{remark}

\begin{remark}
The Dirichlet integral $\int_{\Bn}d^{\alpha}\absgradB{u}^2\,dV$, as in the Euclidean case, corresponds to a \emph{weighted Laplace--Beltrami operator}
\[\DeltaB^{\alpha}{u}=\dfrac{1}{\sqrt{g}}\sum_{i,j=1}^{N}\pd{}{x_i}\left(\sqrt{g}g^{ij}d^{\alpha}\pd{u}{x_j}\right),\]
with $g$ the canonical metric on $\Bn$. It has been considered in \cite{CFM}.
\end{remark}

\begin{remark}
One may put $\alpha=-2a$ and $\beta=-bp$ to state \eqref{ckn} in the form
\[\left(\int_{\Bn}d^{\beta}u^{\frac{2(N+\beta)}{N-2+\alpha}}\,dV\right)^{\frac{N-2+\alpha}{N+\beta}}\leq{C}\int_{\Bn}d^{\alpha}\absgradB{u}^2\,dV,\]
for $-N<\alpha-2\leq\beta\leq\frac{N\alpha}{N-2}$. This motivates the definition of the \emph{critical exponent} $2_\alpha^{\beta}:=\frac{2(N+\beta)}{N-2+\alpha}$.
\end{remark}

As in the Euclidean case, the C--K--N inequalities  \eqref{ckn} contain the Hardy inequality ($a=0$, $b=1$) and the Sobolev inequality ($a=b=0$) on Hyperbolic space as special cases. It is worth mentioning that our proof additionally makes use of the hyperbolic Poincar\'{e} inequality \cite{mancinisandeep} to accommodate the lower order terms. The best constant
\[S(a,b,\Bn):=\inf\limits_{u \in D^{1,2}_a(\Bn)\setminus\{0\}} \dfrac{\int_{\Bn}d^{-2a}\absgradB{u}^2\,dV}{\left(\int_{\Bn}d^{-bp}\abs{u}^p\,dV\right)^{\frac2p}},\]
is known to equal $S(a,b,\R^N)$ in the cases of weighted Hardy inequality ($a<\frac{N-2}{2}$ and $b=a+1$, \cite{ko}) and non-weighted Sobolev inequality ($a=b=0$, \cite{Hebey}).
To our knowledge, the problem of determining $S(a,b,\Bn)$ is largely open. While it is tempting to look for radial extremals, the resulting non-linear ODE cannot be transformed to one with constant coefficients, making the analysis very difficult.


\subsection{Existence and non-existence}


We now concentrate on the study of the equation
\begin{equation} \label{Q}
-\Delta^{\alpha}_{\mathbb{B}^N}u=\lambda d^{\alpha-2}u+d^{\beta}|u|^{q-2}u,\,\, \quad  u\in H^{1}(\mathbb{B}^N),
\end{equation}
where $-N<\alpha-2\leq\beta\leq\frac{N\alpha}{N-2}$, $2 \leq q\leq 2^{\beta}_{\alpha}:=
\frac{2(N+\beta)}{N-2+\alpha},$ and $\lambda < (\frac{N-2+\alpha}{2})^2$. This arises as the Euler--Lagrange equation of an energy functional associated to the weighted Hardy--Sobolev inequality in the hyperbolic space,
\begin{equation}\label{eq:weightedhardysob}
c \left(\int_{\Bn}d^{\beta}u^q\,dV\right)^{2/q}+\lambda \int_{\Bn}d^\alpha\dfrac{u^2}{d^2}\,dV\leq\int_{\Bn}d^{\alpha}\absgradB{u}^2\,dV \quad  \mbox{ for all  } u\in C_c^{\infty}(\Bn).
\end{equation}
Indeed, \eqref{eq:weightedhardysob} can be obtained by interpolating \eqref{ckn} and 
the following weighted Hardy's inequality with sharp constant due to Kombe--\"{O}zaydin \cite{ko}
\begin{equation}\label{eq:weightedhardy}
\left(\frac{N-2+\alpha}{2}\right)^2\int_{\mathbb{B}^N}d^{\alpha}\dfrac{u^2}{d^2}\,dV
\leq\int_{\mathbb{B}^N}d^{\alpha}|\nabla_{\mathbb{B}^N} {u}|^2\,dV,
    \textforall{u}\in{C}_{c}^{\infty}(\Bn)
\end{equation}
via H\"older's inequality. 
An inequality of flavor similar to \eqref{eq:weightedhardysob} is known in \cite{ko}, where the authors showed that for certain smaller exponent $q$, one can in fact take $\lambda=\left(\frac{n-2+\alpha}{2}\right)^2$.

In the absence of the Hardy potential and the weight $d^\alpha$ in the Laplacian (i.e. $\lambda=\alpha=0$), when $\beta>0$, the equation \eqref{Q} of interest originate from the study of stellar structures as proposed by H\'{e}non. Gidas--Spruck \cite{Gidas-Spruck} classified the non-negative solutions of \eqref{Q} in the full Euclidean space $\R^n$ with a Liouville-type theorem so that no non-trivial solution exists for $2<q<2_0^{\beta}=\frac{2(N+\beta)}{N-2}$. In the hyperbolic space, in contrast, He \cite{He} established the existence of solutions in the same range of exponents. This was subsequently generalized by Carri\~{a}o, Miyagaki and the second author \cite{CFM} to the weighted case $\alpha\neq0$, who showed that problem \eqref{Q} possesses a positive radial solution.

In a bounded domain, it is easily shown that ground state solutions exist when $2<q<2^*=\frac{2N}{N-2}$ and $\beta>0$. Ni \cite{ni} found radial solutions when $2<q<2_0^{\beta}$ which, in the case of H\'{e}non equations ($\beta>0$), extend the existence result to Sobolev supercritical exponents. On the other hand, non-existence in star-shaped domains has been proved using generalized versions of the Poho\v{z}aev identity \cite{Pohozaev} in both Hardy and H\'{e}non type problems.

In this paper, we address the remaining cases by considering problem \eqref{Q} which involves the singular potential  (i.e, when $\beta<0$) and  the Hardy term (i.e., $ \lambda d^{\alpha-2}u$). By studying the compactness properties of radial $H^1(\Bn)$ functions, we establish the following existence result.
\begin{thm}\label{TP} Assume that $N\geq 3$, $\beta>\alpha-2>-N$, $\lambda<\left(\frac{N-2+\alpha}{2}\right)^2$  and $q\in(2,2^{\beta}_{\alpha})$. Then, problem \eqref{Q} has (at least) a positive solution.
\end{thm}

\begin{remark}
The critical exponent $2_\alpha^\beta$ is Sobolev critical or supercritical (i.e. $2_\alpha^\beta\geq2^*$) if and only if $\beta\geq\frac{N\alpha}{N-2}$.
\end{remark}

\begin{remark}
We emphasize that with the same arguments as in \cite{He, CFM}, our result is still valid for 
\begin{equation*}
-\Delta^{\alpha}_{\mathbb{B}^N}u-\lambda[d(x)]^{\alpha-2}u=K(d(x))f(u),\;\; u \in H^{1}(\mathbb{B}^N),
\end{equation*}
which involves more general non-linearities 
under suitable growth conditions on $K$ and $f$.

\end{remark}

On the other hand, we establish non-existence of an equation associated to inequality \eqref{ckn} in a star-shaped domain with respect to the origin, in terms of which the distance function $d$ is defined. More precisely, we say that $\Omega\subset\Bn$ is a star-shaped domain if $x\cdot\nu\geq0$ where $\nu$ denotes the outward normal of $\p\Omega$.

For semilinear elliptic equations in hyperbolic space, the Poho\v{z}aev identity has been applied in, for example, \cite{Ganguly-Sandeep,Benguria-Benguria}.

\begin{thm}\label{thm:nonexist}
Let $\Omega\subset\Bn$ be a star-shaped domain and $\alpha,\beta\in\R$ satisfy $-N<\alpha-2\leq\beta$. 
If $p\geq\max\set{2^*,2_\alpha^\beta}$, then there does not exist any non-trivial weak solution to the Dirichlet problem
\begin{equation}\label{eq:CKNextremal}
\begin{cases}
-\DeltaB^\alpha{u}=d^{\beta}\abs{u}^{p-2}u&\textin\Omega\\
u=0&\texton\p\Omega.
\end{cases}
\end{equation}
\end{thm}

We give the proof in Section \ref{sec:nonexist}. The first ingredient is, as expected, a Poho\v{z}aev type identity (Proposition \ref{prop:Pohozaev}). As opposed to the Euclidean case whose corresponding result holds for all $p\geq2_\alpha^\beta$, special care has to be taken because of the lower order terms arising from the derivatives of the distance function and the metric tensor. Under the restrictions $p\geq2^*$, we are able to conclude the proof using the global estimate on $d$,
\[2\abs{x}\leq\log\dfrac{1+\abs{x}}{1-\abs{x}}\leq\dfrac{2\abs{x}}{1-\abs{x}^2}\quad\textforall\abs{x}<1.\]

A careful study of the quantity in Lemma \ref{lem:Pohozaevterms}(2) reveals that it cannot be controlled in a sufficiently large domain whenever $2_\alpha^\beta\leq{p}<2^*$. On this interval of $p$, the question of existence remains open. 

\section{Preliminaries}\label{sec:prelim}

We start by recalling and introducing a suitable function space (on the hyperbolic space) and its properties for the variational principles that will be needed in the sequel. Our
main sources for this section are the papers \cite{bhaktasandeep,mancinisandeep} and the book \cite{raticlife}.


Let $B_1(0)=\{x\in\mathbb{R}^N: |x|<1\}$ be the unit disc in $\mathbb{R}^N$. The Poincar\'{e} ball model of the hyperbolic space, $\Bn$, is the set $B_1(0)$ endowed with the Riemannian metric $g=(g_{ij})$, where
\begin{equation*}(g_{ij}(x))_{i,j=1\ldots N}=(\rho^2(x)\delta_{ij})_{i,j=1\ldots N}= \left\{\begin{array}{ccc}\rho^2&\mbox{ if }&i=j\\ 0&\mbox{ if }&i\neq j\end{array}\right.\end{equation*}
and
\begin{equation*}\begin{array}{rcl} \rho:B_1(0)&\rightarrow&\mathbb{R}\\ x& \mapsto& \frac{2}{1-|x|^2}\end{array}
\end{equation*}
We denote by $g^{ij}$ the components of the inverse matrix of the metric tensor $(g_ {ij})$. Using this notation, we can write the weighted Laplace--Beltrami type operator as
\begin{equation*}
    \begin{array}{ll}
      \displaystyle  -\Delta^{\alpha}_{\mathbb{B}^N}u=-\frac{1}{\rho^N}\sum_{i=1}^N\frac{\partial}{\partial x_i}\left(\rho^{N-2}(d(x))^{\alpha}\frac{\partial u}{\partial x_i}\right),
 \end{array}
\end{equation*}
for $u\in H^1(\mathbb{B}^N),$ where the space  $H^1(\mathbb{B}^N)$ denotes the Sobolev space on $\mathbb{B}^N$ with the metric $g$. We introduce  important quantities in the hyperbolic space which will be used freely in this paper:

\begin{itemize}

\item The hyperbolic gradient $\nabla_{\mathbb{B}^N}$ is given by $$\nabla_{\mathbb{B}^N}=\frac{\nabla}{\rho(x)}.$$

\item The hyperbolic laplacian $\Delta_{\mathbb{B}^N}$ is defined as
$$\Delta^0_{\mathbb{B}^N}=\Delta_{\mathbb{B}^N}= \rho^{-2} \Delta + (N-2)\rho^{-1}\langle{x,\nabla}\rangle.$$

\item
The hyperbolic distance $d_{\mathbb{B}^N}(x,y)$ between $x,y\in$ in $\mathbb{B}^N$ in the Poincar\'e ball model is given by
the formula
$$d_{\mathbb{B}^N}(x,y)=\textrm{Arccosh}\left(1+\frac{2|x-y|^2}{(1-|x|^2)(1-|y|^2)}\right).$$
From this, we immediately obtain for $x\in \mathbb{B}^N$,
$$d(x)=d_{\mathbb{B}^N}(x,0)=\ln \left(\frac{1+|x|}{1-|x|}\right).$$

\item The associated hyperbolic volume element by $dV$ and it is given by $$dV=\left(\frac{2}{1-|x|^2}\right)^Ndx.$$

\end{itemize}
The following continuous embedding holds for  $ p \in [2, \frac{2N}{N-2}]$  when  $ N \geq 3$,  and $ p\geq 2$  when $ N=2$.
$$H^1(\mathbb{B}^N) \hookrightarrow L^p (\mathbb{B}^N).$$
However,  since $B^N$ is a non-compact manifold,  this embedding is  not compact for any  $ p$ in  $ [2, \frac{2N}{N-2}]$ (see \cite{bhaktasandeep}). One
can overcome this difficulty in the subcritical case by restricting to the radial functions. To this end, we define $$H^1_r(\mathbb{B}^N)=\{u\in H^1(\mathbb{B}^N): u \mbox{ is radial}\}.$$
It was shown in \cite[Theorem 3.1]{bhaktasandeep} that the embedding
 $$H_{r}^1(\mathbb{B}^N) \hookrightarrow L^p (\mathbb{B}^N) \ \mbox{ is compact for} \ p \in \left(2, \frac{2N}{N-2}\right), \ N\geq 3.$$

Another important space that will be needed for our variational setting is  the weighted Hyperbolic space. In order to characterize such space, we first recall  the standard norm in the weighted Lebesgue space   $L^q(\mathbb{B}^N,d^{\beta})$  is
$$\|u\|_{\beta,q}=\left(\int_{\mathbb{B}^N}(d(x))^{\beta}|u(x)|^qdV\right)^{\frac{1}{q}}.$$

We consider now the Hilbert space $H_0^1(\mathbb{B}^N;d^{\alpha}dV)$, the completion of $C_c^{\infty}(\mathbb{B}^N)$ with respect to the norm
$$\|u\|_{H_0^1(\mathbb{B}^N;d^{\alpha}dV)}=\left(\int_{\mathbb{B}^N}d^{\alpha}|\nabla_{\mathbb{B}^N}u|^2dV\right)^{\frac{1}{2}}.$$

It is easy to verify  that  $\|u\|_{H_0^1(\mathbb{B}^N;d^{\alpha}dV)}$ is a norm in ${H_0^1(\mathbb{B}^N;d^{\alpha}dV)}.$ Indeed, one can notice that
$\|u\|_{H_0^1(\mathbb{B}^N;d^{\alpha}dV)}=0$ implies that almost everywhere $\nabla_{\mathbb{B}^N}u(x)=0.$ Now, we recall a basic information that the bottom of the spectrum of $-\Delta_{\mathbb{B}^N}$ on $\mathbb{B}^N$ is
\begin{equation}\label{pr}
\lambda_1(-\Delta_{\mathbb{B}^N}):=\inf_{u\in H^1(\mathbb{B}^N)\backslash\set{0}}\dfrac{\int_{\mathbb{B}^N}|\nabla_{\mathbb{B}^N} u|^2dV_{\mathbb{B}^N}}{\int_{\mathbb{B}^N}| u|^2dV_{\mathbb{B}^N}}=\frac{(N-1)^2}{4}.
\end{equation}
It then follows from \eqref{pr} that $u=0$. \\
We finally define the space $H_r^1(\mathbb{B}^N;d^{\alpha}dV)$ as
the subspace of radially symmetric functions endowed with the induced norm
$\|u\|_{H_r^1(\mathbb{B}^N;d^{\alpha}dV)}=\|u\|_{H_0^1(\mathbb{B}^N;d^{\alpha}dV)}.$

\begin{remark}
The hyperbolic sphere with centre $0 \in \mathbb{B}^N$ is also a Euclidean sphere with centre $0 \in \mathbb{B}^N$, therefore  $H_r^1(\mathbb{B}^N;d^{\alpha}dV)$ can also be seen as the subspace of $H_0^1(\mathbb{B}^N;d^{\alpha}dV)$ consisting of Hyperbolic radial
functions.

\end{remark}

We point out an a relation between the distance $d$ and the function $\rho$ that will be useful in Section \ref{sec:nonexist}. For all $r\in[0,1)$ it holds that
\begin{equation*}
2r\leq\log\frac{1+r}{1-r}\leq\frac{2r}{1-r^2}.
\end{equation*}
%


\section{The Caffarelli--Kohn--Nirenberg inequality: Proof of Theorem \ref{thm:CKN}}\label{sec:CKN}

Our idea is simple and has two ingredients. Firstly we will use a change of variable which relates the Caffarell--Kohn--Nirenberg inequality and the Hardy inequality, which is not weighted. Then we interpolate between the following Poincar\'{e}--Sobolev inequality \cite{mancinisandeep}
\begin{equation}\label{eq:ineqPS}
\left(\int_{\Bn}\abs{u}^{\frac{2N}{N-2}}\,dV\right)^{\frac{N-2}{N}}\leq{C}\int_{\Bn}\left(\absgradB{u}^2-\dfrac{(N-1)^2}{4}u^2\right)\,dV
\end{equation}
and the Hardy inequality \cite{ko}
\begin{equation}\label{eq:ineqH}
\dfrac{(N-2)^2}{4}\int_{\Bn}\dfrac{u^2}{d^2}\,dV\leq\int_{\Bn}\absgradB{u}^2\,dV,
\end{equation}
which are valid for all $u\in{C}_c^{\infty}(\Bn)$.

We state a general lemma of changing variables.

\begin{lem}\label{changevar2}
Let $\gamma_1,\gamma_2>0$ and $u=d^{\alpha/2}w$. Then
\begin{equation}\label{eq:changevareq}\begin{split}
&\quad\,\int_{\Bn}\left(\absgradB{u}^2-\gamma_{1}\dfrac{u^2}{d^2}-\gamma_{2}u^2\right)\,dV\\
&=\int_{\Bn}d^{\alpha}\left(\absgradB{w}^2-\left(\dfrac{\alpha(\alpha-2)}{4}+\dfrac{\alpha(N-1)}{2}\dfrac{d}{\rho\abs{x}}+\gamma_{1}\right)\dfrac{w^2}{d^2}-\gamma_{2}w^2\right)\,dV\\
&\qquad-\dfrac{\alpha(N-1)}{2}\int_{\Bn}d^{\alpha-1}w^2\abs{x}\,dV.\\
\end{split}\end{equation}
In particular,
\begin{equation}\label{eq:changevarineq}\begin{split}
\int_{\Bn}\left(\absgradB{u}^2-\gamma_{1}\dfrac{u^2}{d^2}-\gamma_{2}u^2\right)\,dV
&\leq\int_{\Bn}d^\alpha\absgradB{w}^2\,dV.
\end{split}\end{equation}
\end{lem}
The proof is a direct computation and is postponed to the end of this section. Now we give a

\begin{proof}[Proof of Theorem \ref{thm:CKN}]
\begin{enumerate}
Interpolating \eqref{eq:ineqPS} and \eqref{eq:ineqH}, we have for $-2\leq\beta\leq0$,
\begin{equation}\label{eq:interpolate1}
\left(\int_{\Bn}d^{\beta}u^{\frac{2(n+\beta)}{N-2}}\,dV\right)^{\frac{N-2}{n+\beta}}
\leq{C}\int_{\Bn}\absgradB{u}^2\,dV.
\end{equation}
For any $\gamma_1<\frac{(N-2)^2}{4}$, $\gamma_2<\frac{(N-1)^2}{4}$ such that $1-\frac{4}{(N-2)^2}\gamma_1-\frac{4}{(N-1)^2}\gamma_2>0$, we have
\begin{multline}\label{eq:interpolate2}
\left(1-\dfrac{4}{(N-2)^2}\gamma_{1}-\dfrac{4}{(N-1)^2}\gamma_{2}\right)\int_{\Bn}\absgradB{u}^2\,dV\\
\leq\left(\int_{\Bn}\absgradB{u}^2\,dV-\gamma_{1}\int_{\Bn}\dfrac{u^2}{d^2}-\gamma_{2}\int_{\Bn}u^2\,dV\,dV\right).
\end{multline}
Putting together \eqref{eq:changevarineq}, \eqref{eq:interpolate1} and \eqref{eq:interpolate2}, we have for $u=d^{\alpha/2}w$,
\begin{equation*}\begin{split}
\left(\int_{\Bn}d^{\beta+\alpha2_0^\beta/2}w^{2_0^\beta}\,dV\right)^{2/2_0^\beta}
&=\left(\int_{\Bn}d^\beta{u}^{2_0^\beta}\,dV\right)^{2/2_0^\beta}\\
&\leq{C}\int_{\Bn}\left(\absgradB{u}^2-\gamma_1\dfrac{u^2}{d^2}-\gamma_2{u}^2\right)\,dV\\
&\leq{C}\int_{\Bn}d^\alpha\absgradB{w}^2\,dV.
\end{split}\end{equation*}
If we put $\tilde\beta=\beta+\alpha2_0^\beta/2$, then
\[\left(\int_{\Bn}d^{\tilde\beta}w^{2_\alpha^{\tilde\beta}}\,dV\right)^{\frac{2}{2_\alpha^{\tilde\beta}}}\leq{C}\int_{\Bn}d^{\alpha}\absgradB{w}^2\,dV.\]
It remains to rename $\alpha=-2a$ and solve the elementary equations
\[\tilde\beta=-bp,\qquad2_\alpha^{\tilde\beta}=p\]
which gives in particular $p=\frac{2N}{N-2+(b-a)}$.
\end{enumerate}
\end{proof}

%
%
\begin{proof}[Proof of Lemma \ref{changevar2}]
We observe that the non-differentiated terms simply becomes
\begin{equation}\label{eq:nodiff}
-\gamma_2\int_{\Bn}d^{\alpha}\dfrac{w^2}{d^2}\,dV-\gamma_2\int_{\Bn}d^{\alpha}w^2.
\end{equation}
The gradient term is expressed as
\begin{equation}\label{eq:grad}\begin{split}
\int_{\Bn}\absgradB{u}^2\,dV
&=\int_{B}\absgrad{\left(d^{\frac{\alpha}{2}}w\right)}^2\rho^{N-2}\,dx\\
&=\int_{B}\abs{d^{\frac{\alpha}{2}}\nabla{w}+\dfrac{\alpha}{2}wd^{\frac{\alpha}{2}-1}\nabla{d}}^2\rho^{N-2}\,dx\\
&=\int_{B}\left(d^{\alpha}\absgrad{w}^2+\dfrac{\alpha^2}{4}w^2d^{\alpha-2}\absgrad{d}^2+\alpha{w}\nabla{w}\cdot{d}^{\alpha-1}\nabla{d}\right)\rho^{N-2}\,dx\\
&=\int_{\Bn}d^{\alpha}\left(\absgradB{w}^2+\dfrac{\alpha^2}{4}\dfrac{w^2}{d^2}\right)\,dV+\dfrac{1}{2}\int_{B}\nabla{w^2}\cdot\nabla{d^{\alpha}}\rho^{N-2}\,dx.
\end{split}\end{equation}
For the last integral, we perform an integration by parts to yield
\begin{equation}\label{eq:byparts1}\begin{split}
\dfrac{1}{2}\int_{B}\nabla{w^2}\cdot\nabla{d^{\alpha}}\rho^{N-2}\,dx
&=-\dfrac{1}{2}\int_{B}w^2\nabla\cdot(\rho^{N-2}\nabla{d}^{\alpha})\,dx\\
&=-\dfrac{\alpha}{2}\int_{B}w^2\nabla\cdot\left(\rho^{N-1}d^{\alpha-1}\dfrac{x}{\abs{x}}\right)\,dx.
\end{split}\end{equation}
Observing that
\begin{equation*}\begin{split}
\nabla\cdot\left(\rho^{N-1}\dfrac{x}{\abs{x}}\right)
&=(N-1)\rho^{N-2}\rho^2{x}\cdot\dfrac{x}{\abs{x}}+\rho^{N-1}\dfrac{N-1}{\abs{x}}\\
&=(N-1)\rho^{N}\left(\abs{x}+\dfrac{1}{\rho\abs{x}}\right)\\
&=\dfrac{N-1}{2}\rho^{N}\dfrac{1+\abs{x}^2}{\abs{x}},
\end{split}\end{equation*}
we may expand the divergence in \eqref{eq:byparts1} as
\begin{equation}\label{eq:byparts2}\begin{split}
&\quad\,\dfrac{1}{2}\int_{B}\nabla{w^2}\cdot\nabla{d^{\alpha}}\rho^{N-2}\,dx\\
&=-\dfrac{\alpha}{2}\int_{B}w^2\nabla\cdot\left(\rho^{N-1}d^{\alpha-1}\dfrac{x}{\abs{x}}\right)\,dx\\
&=-\dfrac{\alpha}{2}\int_{B}w^2\left(d^{\alpha-1}\dfrac{N-1}{2}\rho^{N}\dfrac{1+\abs{x}^2}{\abs{x}}+\rho^{N-1}\dfrac{x}{\abs{x}}(\alpha-1)d^{\alpha-2}\rho\dfrac{x}{\abs{x}}\right)\,dx\\
&=-\dfrac{\alpha(N-1)}{4}\int_{B}d^{\alpha-1}w^2\rho^{N}\dfrac{1+\abs{x}^2}{\abs{x}}\,dx-\dfrac{\alpha(\alpha-1)}{2}\int_{B}d^{\alpha-2}w^2\rho^{N}\,dx.
\end{split}\end{equation}
Combining \eqref{eq:nodiff}, \eqref{eq:grad} and \eqref{eq:byparts2}, we have
\begin{equation*}\begin{split}
&\quad\,\int_{\Bn}\left(\absgradB{u}^2-\gamma_1\dfrac{u^2}{d^2}-\gamma_2u^2\right)\,dV\\
&=\int_{\Bn}d^{\alpha}\left(\absgradB{w}^2+\left(\dfrac{\alpha^2}{4}-\gamma_1\right)\dfrac{w^2}{d^2}-\gamma_2w^2\right)\,dV+\dfrac{1}{2}\int_{B}\nabla{w^2}\cdot\nabla{d^{\alpha}}\rho^{N-2}\,dx\\
&=\int_{\Bn}d^{\alpha}\left(\absgradB{w}^2+\left(\dfrac{\alpha^2}{4}-\dfrac{\alpha(\alpha-1)}{2}-\gamma_1\right)\dfrac{w^2}{d^2}-\gamma_2w^2\right)\,dV\\
&\qquad-\dfrac{\alpha(N-1)}{4}\int_{B}d^{\alpha-1}w^2\rho^{N}\dfrac{1+\abs{x}^2}{\abs{x}}\,dx\\
&=\int_{\Bn}d^{\alpha}\left(\absgradB{w}^2-\left(\dfrac{\alpha(\alpha-2)}{4}+\gamma_1\right)\dfrac{w^2}{d^2}-\gamma_2w^2\right)\,dV\\
&\qquad-\dfrac{\alpha(N-1)}{4}\int_{\Bn}d^{\alpha-1}w^2\dfrac{1+\abs{x}^2}{\abs{x}}\,dV\\
&=\int_{\Bn}d^{\alpha}\left(\absgradB{w}^2-\left(\dfrac{\alpha(\alpha-2)}{4}+\dfrac{\alpha(N-1)}{2}\dfrac{d}{\rho\abs{x}}+\gamma_{1}\right)\dfrac{w^2}{d^2}-\gamma_{2}w^2\right)\,dV\\
&\qquad-\dfrac{\alpha(N-1)}{2}\int_{\Bn}d^{\alpha-1}w^2\abs{x}\,dV.\\
\end{split}\end{equation*}
This proves \eqref{eq:changevareq}. To prove \eqref{eq:changevarineq}, we note that the second integral in \eqref{eq:changevareq} can be absorbed by the ``Cauchy--Schwarz inequality with $\eps$'', namely
\[\begin{split}
\int_{\Bn}d^{\alpha-1}w^2\abs{x}\,dV
&\leq\eps\int_{\Bn}d^\alpha\dfrac{w^2}{d^2}\,dV+C_\eps\int_{\Bn}d^{\alpha}w^2\abs{x}^2\,dV
\end{split}\]
The proof is completed by taking $0<\eps\leq\gamma_2$ and using \eqref{eq:weightedhardy}.

\end{proof}

\section{A compact embedding}
In this section we prove a weighted compact Sobolev embedding result, which extends a former result of Ni \cite{ni} made for  an  unit ball in $\mathbb{R}^N$.

Let $H_0^1(\mathbb{B}^N;d^{\alpha}dV)$ be the completion of $C_c^{\infty}(\mathbb{B}^N)$ with respect to the norm
$$\|u\|_{H_0^1(\mathbb{B}^N;d^{\alpha}dV)}=\left(\int_{\mathbb{B}^N}d^{\alpha}|\nabla_{\mathbb{B}^N}u|^2dV\right)^{\frac{1}{2}}.$$
 We denote by $ H_r^1(\mathbb{B}^N;d^{\alpha}dV)$
the subspace of radially symmetric functions endowed with the induced norm
$\|u\|_{H_r^1(\mathbb{B}^N;d^{\alpha}dV)}=\|u\|_{H_0^1(\mathbb{B}^N;d^{\alpha}dV)}.$



By Theorem \ref{thm:CKN}, the embedding
$$H^1(\Bn;d^{\alpha}\,dV)\hookrightarrow{L}^{p}(\Bn;d^{\beta}\,dV)$$
is continuous for $2\leq{p}\leq2_\alpha^\beta$. We will show that it is in fact compact for $2<p<2_\alpha^\beta$, when restricted to radial functions. When $\beta>0$, a proof can be found in \cite[Lemma 1]{CFM}, in the spirit of Strauss \cite{Strauss}. We present a unified and simplified proof.


In the following, $f\sim{g}$ means that $f/g$ is a bounded positive function; and $\log^{a}(x)$ means $(\log{x})^{a}$.

\begin{prop}[Compact embedding]\label{compactembed}
Let $n\geq2$ and $\beta>\alpha-2>-N$. The embedding $$H^1_r(\Bn;d^{\alpha}\,dV)\hookrightarrow{L}^{p}(\Bn;d^{\beta}\,dV)$$ is compact when $2<p<2_\alpha^\beta$, or when $p=2$ and $\beta<\alpha-1$.
\end{prop}

\begin{proof}
Let $(u_m)\subset H^1_r(\Bn;d^{\alpha}\,dV)$ be a bounded sequence. Up to a subsequence, we can assume $u_m\rightharpoonup{u}$ weakly in $H^1_r(\Bn;d^{\alpha}\,dV)$ and $u_m\to{u}$ a.e. in $\Bn$. It suffices to show that
$$\int_{\Bn}d^{\beta}\abs{u_m}^p\,dV\to\int_{\Bn}d^{\beta}\abs{u}^p\,dV,$$
that is,
$$\int_0^1d^{\beta}\rho^{N}r^{N-1}\abs{u_m(r)}^p\,dr\to\int_0^1d^{\beta}\rho^{N}r^{N-1}\abs{u(r)}^p\,dr.$$
By Lebesgue dominated convergence theorem, it suffices to verify that the integrand on the left is dominated by some function $h\in{L^1}([0,1])$.

Proceeding as above, we have
\begin{equation*}\begin{split}
d^{\beta}\rho^{N}r^{N-1}\abs{u_m(r)}^p
\leq&\;{C} \norm[H^1_r(\Bn;d^{\alpha}\,dV)]{u_m}^{p}r^{\frac{N-2+\alpha}{2}(2_\alpha^\beta-p)-1}(1-r)^{\frac{N-1}{2}(p-2)-1}\times \\ &\log^{\beta-\frac{\alpha{p}}{2}}\max\set{\dfrac{1}{1-r},e}\\
\leq&\;{C}r^{\frac{N-2+\alpha}{2}(2_\alpha^\beta-p)-1}(1-r)^{\frac{N-1}{2}(p-2)-1}\log^{\beta-\frac{\alpha{p}}{2}}\max\set{\dfrac{1}{1-r},e},
\end{split}\end{equation*}
since $(u_m)$ is bounded in $H^1_r(\Bn;d^{\alpha}\,dV)$. Let $h$ denotes the right hand side. When $2<p<2_\alpha^\beta$, we can ignore the logarithmic factor and we see that $h$ is integrable near both $r=0$ and $r=1$. When $p=2$, we have
$$h(r)\sim(1-r)^{-1}\log^{\beta-\alpha}\dfrac{1}{1-r}\qquad\text{near}\quad{r=1},$$
or, by the change of variable $t=\log\dfrac{1}{1-r}$, $\dfrac{dr}{1-r}=dt$,
\begin{equation*}\begin{split}
\int_{1-1/e}^{1}h(r)\,dr
&\leq{C}\int_{1-1/e}^{1}\dfrac{1}{1-r}\log^{\beta-\alpha}\dfrac{1}{1-r}\,dr\\
&\leq{C}\int_{1}^{\infty}t^{\beta-\alpha}\,dt\\
&<\infty,
\end{split}\end{equation*}
provided that $\beta-\alpha<-1$.
\end{proof}

\begin{remark}\label{r1}
Note that inequality (\ref{eq:weightedhardy}) asserts that $H_0^1(\mathbb{B}^N;d^{\alpha}dV)$ is embedded in the weighted space $L^2(\mathbb{B}^N, d^{\alpha-2} dV)$ and that this embedding is continuous. If $\lambda < \left(\frac{N-2+\alpha}{2} \right)^2$, it follows from (\ref{eq:weightedhardy}) that
$$ \|w\|  := \left( \int_{\mathbb{B}^N}d^{\alpha}|\nabla_{\mathbb{B}^N} u|^2dV - \lambda \int_{\mathbb{B}^N} d^{\alpha - 2}|u|^2 dV \right)^\frac{1}{2} $$
is well-defined on $H_r^1(\mathbb{B}^N;d^{\alpha}dV)$ and is equivalent to the norm $\|\cdot\|_{H_r^1(\mathbb{B}^N;d^{\alpha}dV)}$.

\end{remark}

\section{Existence of the weighted Hardy--H\'{e}non equation: Proof of Theorem \ref{TP}}\label{S5}

The proof of the existence result rely on the Mountain Pass Theorem due to Ambrosetti and Rabinowitz \cite{AR}. Associated with the problem $\eqref{Q}$ we define the functional $I:H_r^1(\mathbb{B}^N;d^{\alpha}dV) \longrightarrow \mathbb{R}$ given by
$$I(u)=\frac{1}{2} \|u\|^2- \frac{1}{q}\int_{\mathbb{B}^N} d^{\beta}|u|^qdV,$$
which is $C^1$  and its derivative  is,  for all $u,v \in H_r^1(\mathbb{B}^N;d^{\alpha}dV)$, given by

\begin{equation*}\begin{split}
I'(u)v&= \int_{\mathbb{B}^N}  d^{\alpha} \nabla_{\mathbb{B}^N} u \cdot \nabla_{\mathbb{B}^N} v \  dV- \lambda\int_{\mathbb{B}^N} d^{\alpha -2} u v dV\\
& - \int_{\mathbb{B}^N} d^{\beta}|u|^{q-2}uvdV.
\end{split}\end{equation*}


We are going to prove that $I$ verifies the Mountain Pass Geometry conditions.
\begin{lemma}\label{le5} Let $N\geq 3$, $\alpha> 2-N$, $\beta<0$ and $q\in(2,2^{\beta}_{\alpha})$. Then
\medskip
\begin{itemize}
\item [i)] there exist positive constants $\rho, \gamma$ such that
$I(u)\geq \gamma, \ \mbox{for all} \ \ \|u\|_{H_r^1(\mathbb{B}^N;d^{\alpha}dV)}=\rho;$
\item [ii)] there exist a constant $R > \rho$ and $e \in H_r^1(\mathbb{B}^N;d^{\alpha}dV)$  with $ \|e\|_{H_r^1(\mathbb{B}^N;d^{\alpha}dV)} >R,$ verifying $I(e)\leq 0.$
\end{itemize}
\end{lemma}
\begin{proof}
By 
Theorem \ref{thm:CKN} and Remark \ref{r1}, we get
\begin{equation*}\begin{split}
I(u)&= \frac{1}{2} \|u\|^2- \frac{1}{q}\int_{\mathbb{B}^N} d^{\beta}|u|^qdV\\
&\geq  C_1  \|u\|_{H_r^1(\mathbb{B}^N;d^{\alpha}dV)}^2 -C_2\|u\|^q_{H_r^1(\mathbb{B}^N;d^{\alpha}dV)}.
\end{split}\end{equation*}
Since $q>2$, there exists $\gamma>0$ such that
$$I(u)\geq \gamma, \mbox{ for } \|u\|_{H_r^1(\mathbb{B}^N;d^{\alpha}dV)}=\rho \ \ \mbox{sufficiently small.}$$
This proves $(i).$

Now, choose any $ u \in H_r^1(\mathbb{B}^N;d^{\alpha}dV)\setminus\{0\}$, then
$$I(tu)= \frac{t^2}{2} \|u\|^2 - \frac{ |t|^{q} }{q}\int_{\mathbb{B}^N} d^{\beta} |u|^{q} dV.$$
Therefore
$$I(tu) \longrightarrow -\infty, \  \mbox{as}\   t \rightarrow \infty.$$
This proves $(ii).$

\end{proof}

\begin{lemma}\label{ps}
According to our previous notation, $I$ satisfies the Palais-Smale condition $(PS)$,  that is,
if whenever  $\{u_n\} $ is a sequence in $ H_r^1(\mathbb{B}^N;d^{\alpha}dV)$ such that
$$I(u_n) \ \mbox{is bounded, and} \  \ I'(u_n) \rightarrow 0, \ \mbox{as} \ n\rightarrow \infty,$$
then   $\{u_n\} $  has  a convergent subsequence in $H_r^1(\mathbb{B}^N;d^{\alpha}dV).$
\end{lemma}
\begin{proof}
For some positive constant $M>0,$ we get
\begin{eqnarray*}
M+ \|u_n\|_{H_r^1(\mathbb{B}^N;d^{\alpha}dV)}&\geq& I(u_n)-\frac{1}{q} I'(u_n) u_n \\
&=&C(\frac{1}{2}- \frac{1}{q})\|u_n\|_{H_r^1(\mathbb{B}^N;d^{\alpha}dV)}^2.
\end{eqnarray*}
This implies that $\|u_n\|_{H_r^1(\mathbb{B}^N;d^{\alpha}dV)}$ is bounded.
So that,  there exists $u \in H_r^1(\mathbb{B}^N;d^{\alpha}dV)$ such that, passing to a subsequence if necessary,
$$u_n \rightharpoonup u \ \mbox{weakly in }\ H_r^1(\mathbb{B}^N;d^{\alpha}dV) \ \mbox{and pointwise},  \mbox{ as } \ n \rightarrow \infty.$$
By Proposition \ref{compactembed}, we get  $u_n \rightarrow u, \ \mbox{in} \ L^{q}(\mathbb{B}^N;d^{\beta}dV), \mbox{ as } \ n \rightarrow \infty, \ q\in (2,\frac{2(N+\beta)}{N-2 +\alpha})$.
Since $$  C \|u_n-u\|_{H_r^1(\mathbb{B}^N;d^{\alpha}dV)}^2=I'(u_n)(u_n-u)-I'(u)(u_n-u)+o(1)=o(1)  \; \mbox{ as } \;  n \rightarrow \infty,$$ we have
$$ u_n \longrightarrow u, \ \mbox{in} \ H_r^1(\mathbb{B}^N;d^{\alpha}dV), \ \mbox{as} \  n \rightarrow \infty.$$
By Mountain Pass Theorem, there exists a solution $u \in H_r^1(\mathbb{B}^N;d^{\alpha}dV).$ By testing $ u^{-}=\max\{ -u,0\}$ in $I'(u)u^{-}=0,$ we reach that $u^{-}=0.$ That is, $ u=u^{+}=\max\{u, 0\}.$
By invoking a version of the  maximum principle due to Antonini,  Mugnai and Pucci \cite{antonini}, we can conclude the positivity of the solution.

 Now, we will use standard arguments to conclude that the function $u$ is a critical point of $I$ in $H_r^1(\mathbb{B}^N;d^{\alpha}dV)$. Since  $u$ is a critical point of $I$ in $H_r^1(\mathbb{B}^N;d^{\alpha}dV)$,  then $I'(u) v_r=0, $ \ for all $v_r \in H_r^1(\mathbb{B}^N;d^{\alpha}dV)$. Notice that $H_r^1(\mathbb{B}^N;d^{\alpha}dV)$ is a closed subspace of the Hilbert space $H_r^1(\mathbb{B}^N;d^{\alpha}dV)$, so that, we can write
$$H_r^1(\mathbb{B}^N;d^{\alpha}dV)=H_r^1(\mathbb{B}^N;d^{\alpha}dV)\oplus H_r^1(\mathbb{B}^N;d^{\alpha}dV)^{\perp}.$$
Therefore,  for any $v \in H_r^1(\mathbb{B}^N;d^{\alpha}dV)$ we can split it as  $$v= v_r + v^{\perp}, \  \ \mbox{with} \ \  v_r \in H_r^1(\mathbb{B}^N;d^{\alpha}dV) \ \mbox{and} \  \ v^{\perp} \in   H_r^1(\mathbb{B}^N;d^{\alpha}dV)^{\perp}.$$
On the other hand, since $H_r^1(\mathbb{B}^N;d^{\alpha}dV)$ is also a Hilbert space, we can identify, through the duality, $I'(u)$ with an element in $H_r^1(\mathbb{B}^N;d^{\alpha}dV)$. Then $$\langle I'(u), v^{\perp}\rangle_{H_r^1(\mathbb{B}^N;d^{\alpha}dV)} =0.$$
Hence, for all $ v \in H_r^1(\mathbb{B}^N;d^{\alpha}dV)$
$$ \langle I'(u), v\rangle_{H_r^1(\mathbb{B}^N;d^{\alpha}dV)}=\langle I'(u),v_r\rangle_{H_r^1(\mathbb{B}^N;d^{\alpha}dV)}+ \langle I'(u), v^{\perp}\rangle_{H_r^1(\mathbb{B}^N;d^{\alpha}dV)}=0.$$
This proves that the function $u$ is a critical point of $I$ in $H_r^1(\mathbb{B}^N;d^{\alpha}dV)$. Therefore, $u$ is a positive solution of \eqref{Q}.
\end{proof}

\section{Non-existence in Sobolev-supercritical case: Proof of Theorem \ref{thm:nonexist}}
\label{sec:nonexist}

We employ a Poho\v{z}aev type identity, taking into account the extra terms arising from the hyperbolic metric and the weight. Specifically, we test \eqref{eq:CKNextremal} against $\nabla{u}\cdot{x}$ as well as a weighted version of $u$, namely $\frac{\nabla\cdot(d^{\beta}\rho^{N}x)}{d^{\beta}\rho^{N}}u$.

\begin{prop}\label{prop:Pohozaev}
If $u$ solves \eqref{eq:CKNextremal}, then
\begin{multline}\label{eq:Pohozaev}
-\dfrac12\int_{\p\Omega}d^{\alpha}\rho^{n-2}\absgrad{u}^2x\cdot\nu\,d\sigma(x)\\
=\int_{\Omega}d^{\alpha}\absgradB{u}^2\left(-1+\dfrac{\nabla\cdot(d^\alpha\rho^{N-2}x)}{2d^{\alpha}\rho^{N-2}}-\dfrac{\nabla\cdot(d^{\beta}\rho^{N}x)}{pd^{\beta}\rho^{N}}\right)\,dx
+\dfrac{1}{2p}\int_{\Omega}u^2\DeltaB^\alpha\left(\dfrac{\nabla\cdot(d^\beta\rho^{N}x)}{d^\beta\rho^{N}}\right)\,dx
\end{multline}
\end{prop}

The proof is completed by a straightforward computation showing the positivity of the two brackets on the right hand side of \eqref{eq:Pohozaev}.

\begin{lem}\label{lem:Pohozaevterms}
We have the following.
\begin{enumerate}
\item For $\abs{x}<1$,
\[\dfrac{2\abs{x}}{1+\abs{x}^2}\leq2\abs{x}\leq{d(x)}\leq\dfrac{2\abs{x}}{1-\abs{x}^2}.\]
This implies, in particular, that the non-negative quantities $A=1+\rho\abs{x}^2-\frac{\rho\abs{x}}{d}$ and $B=\frac{\rho\abs{x}}{d}$ satisfy
\[{d}\rho\abs{x}\geq{A}\quad\textand\quad{d}\rho\abs{x}B-(B^2-1)\geq0.\]
\item For $p\geq\frac{2N}{N-2}\geq2_\alpha^{\beta}$,
\[-1+\dfrac{\nabla\cdot(d^\alpha\rho^{N-2}x)}{2d^{\alpha}\rho^{N-2}}-\dfrac{\nabla\cdot(d^{\beta}\rho^{N}x)}{pd^{\beta}\rho^{N}}\geq0.\]
\item For $N\geq\alpha-1$ and $\beta>-N$,
\[\DeltaB^\alpha\left(\dfrac{\nabla\cdot(d^\beta\rho^{N}x)}{d^\beta\rho^{N}}\right)\geq0.\]
\end{enumerate}
\end{lem}

\begin{proof}[Proof of Theorem \ref{thm:nonexist}]
This follows immediately from Proposition \ref{prop:Pohozaev} and Lemma \ref{lem:Pohozaevterms} which force all the integrals in \eqref{eq:Pohozaev} to vanish.
\end{proof}

To prove Proposition \ref{prop:Pohozaev} we recall the following facts.

\begin{lem}\label{lem:Pohozaevbasic}
We have the following.
\begin{enumerate}
\item A Bochner-type identity
\[\nabla{u}\cdot\nabla(\nabla{u}\cdot{x})=\absgrad{u}^2+\dfrac12\nabla\absgrad{u}^2\cdot{x}.\]
\item The boundary derivatives
\[\nabla{u}\cdot\nu=\absgrad{u}\quad\textand\quad\nabla{u}\cdot{x}=(x\cdot\nu)\absgrad{u}.\]
\end{enumerate}
\end{lem}
\begin{proof}
(1) is a straightforward computation. For (2) we observe that from the Dirichlet boundary condition, $\nu=\frac{\nabla{u}}{\absgrad{u}}$.
\end{proof}

\begin{proof}[Proof of Proposition \ref{prop:Pohozaev}]
The equation \eqref{eq:CKNextremal} can be rewritten as
\begin{equation}\label{eq:CKNextremal2}
-\nabla\cdot(d^{\alpha}\rho^{N-2}\nabla{u})=d^{\beta}\rho^{N}\abs{u}^{p-2}u\quad\textin\Omega.
\end{equation}
Multiplying both sides by $\nabla{u}\cdot{x}$ and integrating over $\Omega$, we find
\begin{equation}\label{eq:poho1}
-\int_{\p\Omega}d^{\alpha}\rho^{N-2}(\nabla{u}\cdot\nu)(\nabla{u}\cdot{x})\,d\sigma(x)
    +\int_{\Omega}d^{\alpha}\rho^{N-2}\nabla{u}\cdot\nabla(\nabla{u}\cdot{x})\,dx
        =\int_{\Omega}d^{\beta}\rho^{N}\abs{u}^{p-2}u\nabla{u}\cdot{x}\,dx.
\end{equation}
By Lemma \ref{lem:Pohozaevbasic}(2), the boundary integral is equal to
\[-\int_{\p\Omega}d^{\alpha}\rho^{N-2}\absgrad{u}^2(x\cdot\nu)\,d\sigma(x)\]
as appears in the left hand side of \eqref{eq:Pohozaev}.

Following the standard argument we will integrate by parts several times. Using Lemma \ref{lem:Pohozaevbasic}(1), the integral on $\Omega$ of the left hand side of \eqref{eq:poho1} is equal to
\[\begin{split}
&\quad\,\int_{\Omega}d^{\alpha}\rho^{N-2}\nabla{u}\cdot\nabla(\nabla{u}\cdot{x})\,dx\\
&=\int_{\Omega}d^{\alpha}\rho^{N-2}\absgrad{u}^2\,dx
    +\dfrac12\int_{\Omega}d^{\alpha}\rho^{N-2}\nabla\absgrad{u}^2\cdot{x}\,dx\\
&=\int_{\Omega}d^{\alpha}\rho^{N-2}\absgrad{u}^2\,dx
    +\dfrac12\int_{\Omega}\nabla\cdot\left(d^{\alpha}\rho^{N-2}\absgrad{u}^2\cdot{x}\right)\,dx
    -\dfrac12\int_{\Omega}\absgrad{u}^2\nabla\cdot\left(d^\alpha\rho^{N-2}x\right)\,dx\\
&=\int_{\Omega}d^{\alpha}\rho^{N-2}\absgrad{u}^2\,dx
    +\dfrac12\int_{\p\Omega}d^{\alpha}\rho^{N-2}\absgrad{u}^2(x\cdot\nu)\,d\sigma(x)
    -\dfrac12\int_{\Omega}\absgrad{u}^2\nabla\cdot\left(d^\alpha\rho^{N-2}x\right)\,dx\\
\end{split}\]
For the right hand side of \eqref{eq:poho1} we use the Dirichlet boundary condition to obtain
\[\begin{split}
\int_{\Omega}d^{\beta}\rho^{N}\abs{u}^{p-1}u\nabla{u}\cdot{x}\,dx
&=\dfrac{1}{p}\int_{\Omega}d^{\beta}\rho^{N}\nabla(\abs{u}^p)\cdot{x}\,dx\\
&=\dfrac{1}{p}\int_{\Omega}\nabla\cdot\left(d^\beta\rho^{N}\abs{u}^p{x}\right)\,dx
    -\dfrac1p\int_{\Omega}\abs{u}^p\nabla\cdot\left(d^{\beta}\rho^{N}x\right)\,dx\\
&=\dfrac{1}{p}\int_{\p\Omega}d^\beta\rho^{N}\abs{u}^p(x\cdot\nu)\,dx
    -\dfrac1p\int_{\Omega}\abs{u}^p\nabla\cdot\left(d^{\beta}\rho^{N}x\right)\,dx\\
&=-\dfrac1p\int_{\Omega}\abs{u}^p\nabla\cdot\left(d^{\beta}\rho^{N}x\right)\,dx.
\end{split}\]
\end{proof}
Putting things together, \eqref{eq:poho1} becomes
\begin{multline}\label{eq:poho2}
-\dfrac12\int_{\p\Omega}d^{\alpha}\rho^{N-2}\absgrad{u}^2(x\cdot\nu)\,d\sigma(x)\\
=-\int_{\Omega}d^{\alpha}\rho^{N-2}\absgrad{u}^2\,dx
    +\dfrac12\int_{\Omega}\absgrad{u}^2\nabla\cdot\left(d^\alpha\rho^{N-2}x\right)\,dx
    -\dfrac1p\int_{\Omega}\abs{u}^p\nabla\cdot\left(d^{\beta}\rho^{N}x\right)\,dx.
\end{multline}
It remains to rewrite the last integral. To this end we test \eqref{eq:CKNextremal2} with $-\frac1p\frac{\nabla\cdot(d^{\beta}\rho^{N}x)}{d^{\beta}\rho^{N}}u$ which yields, using again that $u=0$ on $\p\Omega$,
\begin{equation}\label{eq:poho3}\begin{split}
&\quad\,-\dfrac1p\int_{\Omega}\abs{u}^p\nabla\cdot\left(d^{\beta}\rho^{N}x\right)\,dx\\
&=\dfrac1p\int_{\Omega}\dfrac{\nabla\cdot(d^{\beta}\rho^{N}x)}{d^{\beta}\rho^{N}}u\nabla\cdot(d^{\alpha}\rho^{N-2}\nabla{u})\,dx\\
&=\dfrac1p\int_{\Omega}\nabla\cdot\left(\dfrac{\nabla\cdot(d^{\beta}\rho^{N}x)}{d^{\beta}\rho^{N}}d^{\alpha}\rho^{N-2}u\nabla{u}\right)\,dx
    -\dfrac1p\int_{\Omega}d^{\alpha}\rho^{N-2}\nabla{u}\cdot\nabla\left(\dfrac{\nabla\cdot(d^{\beta}\rho^{N}x)}{d^{\beta}\rho^{N}}u\right)\,dx\\
&=\dfrac1p\int_{\p\Omega}\dfrac{\nabla\cdot(d^{\beta}\rho^{N}x)}{d^{\beta}\rho^{N}}d^{\alpha}\rho^{N-2}u\nabla{u}\cdot\nu\,d\sigma(x)\\
&\qquad-\dfrac1p\int_{\Omega}d^{\alpha}\rho^{N-2}\absgrad{u}^2\left(\dfrac{\nabla\cdot(d^{\beta}\rho^{N}x)}{d^{\beta}\rho^{N}}\right)\,dx
    -\dfrac1p\int_{\Omega}d^{\alpha}\rho^{N-2}u\nabla{u}\cdot\nabla\left(\dfrac{\nabla\cdot(d^{\beta}\rho^{N}x)}{d^{\beta}\rho^{N}}\right)\,dx.
\end{split}\end{equation}
Here the boundary integral vanishes. In the last integral we write $u\nabla{u}=\frac12\nabla(u^2)$ and integrate by parts, giving
\begin{equation}\label{eq:poho4}\begin{split}
&\quad\,-\dfrac{1}{2p}\int_{\Omega}d^{\alpha}\rho^{N-2}\nabla(u^2)\cdot\nabla\left(\dfrac{\nabla\cdot(d^{\beta}\rho^{N}x)}{d^{\beta}\rho^{N}}\right)\,dx\\
&=-\dfrac{1}{2p}\int_{\p\Omega}d^{\alpha}\rho^{N-2}u^2\nabla\left(\dfrac{\nabla\cdot(d^{\beta}\rho^{N}x)}{d^{\beta}\rho^{N}}\right)\cdot\nu\,d\sigma(x)
    +\dfrac{1}{2p}\int_{\Omega}u^2\nabla\cdot\left(d^{\alpha}\rho^{N-2}\left(\dfrac{\nabla\cdot(d^{\beta}\rho^{N}x)}{d^{\beta}\rho^{N}}\right)\right)\,dx\\
&=\dfrac{1}{2p}\int_{\Omega}u^2\DeltaB^{\alpha}\left(\dfrac{\nabla\cdot(d^{\beta}\rho^{N}x)}{d^{\beta}\rho^{N}}\right)\,dx.
\end{split}\end{equation}
The proposition now follows from \eqref{eq:poho2}, \eqref{eq:poho3} and \eqref{eq:poho4}.

We conclude the section by supplying the calculations of the explicit functions.

\begin{proof}[Proof of Lemma \ref{lem:Pohozaevterms}]
\begin{enumerate}
\item The first global estimate is well-known and can be proved by the equality at $r=\abs{x}=0$ and the corresponding inequalities of the radial derivatives
    \[(2r)'\leq{d'(r)}=\dfrac{2}{1-r^2}\leq\left(\dfrac{2r}{1-r^2}\right)'=\dfrac{2(1+r^2)}{(1-r^2)^2}.\]
    Then we have
    \[\begin{split}
    d\abs{x}\rho-A
    &=\rho\left(\abs{x}d-\dfrac{1-\abs{x}^2}{2}-\abs{x}^2+\dfrac{\abs{x}}{d}\right)\\
    &=\rho\left(\abs{x}d+\dfrac{\abs{x}}{d}-\dfrac{1+\abs{x}^2}{2}\right)\\
    &\geq\rho\left(2\abs{x}^2+\dfrac{1-\abs{x}^2}{2}-\dfrac{1+\abs{x}^2}{2}\right)\\
    &=\rho\abs{x}^2\geq0
    \end{split}\]
    and
    \[\begin{split}
    d\rho\abs{x}B-B^2+1
    &=\rho^2\abs{x}^2\left(1-\dfrac{1}{d^2}+\dfrac{1}{\rho^2\abs{x}^2}\right)\\
    &\geq0
    \end{split}\]
    for
    \[d^2\geq\dfrac{1}{1+\frac{1}{\rho^2\abs{x}^2}}=\dfrac{4\abs{x}^2}{4\abs{x}^2+(1-\abs{x}^2)^2}=\left(\dfrac{2\abs{x}}{1+\abs{x}^2}\right)^2\]
    which is known to be true.
\item We compute, using part (1),
\[\begin{split}
&\quad\,-1+\dfrac{\nabla\cdot(d^\alpha\rho^{N-2}x)}{2d^{\alpha}\rho^{N-2}}-\dfrac{\nabla\cdot(d^{\beta}\rho^{N}x)}{pd^{\beta}\rho^{N}}\\
&=-1+\dfrac12\left(N+(N-2)\rho\abs{x}^2+\alpha\rho\dfrac{\abs{x}}{d}\right)-\dfrac1p\left(N+N\rho\abs{x}^2+\beta\rho\dfrac{\abs{x}}{d}\right)\\
&=\left(\dfrac{N-2}{2}-\dfrac{N}{p}\right)(1+\rho\abs{x}^2)+\left(\dfrac{\alpha}{2}-\dfrac{\beta}{p}\right)\rho\dfrac{\abs{x}}{d}\\
&\geq\left(\dfrac{N-2}{2}-\dfrac{N}{p}+\dfrac{\alpha}{2}-\dfrac{\beta}{p}\right)\rho\dfrac{\abs{x}}{d}\qquad\textsince{p}\geq\dfrac{2N}{N-2}\\
&=\left(\dfrac{N-2+\alpha}{2}-\dfrac{N+\beta}{p}\right)\rho\dfrac{\abs{x}}{d}\\
&\geq0\qquad\textsince{p}\geq2_\alpha^\beta.
\end{split}\]
\item From
\[\dfrac{\nabla\cdot(d^{\beta}\rho^{N}x)}{d^{\beta}\rho^{N}}=N+N\rho\abs{x}^2+\beta\rho\dfrac{\abs{x}}{d},\]
it suffices to compute the functions when weighted Laplace--Beltrami operator is acted upon $\rho\abs{x}^2$ and $\rho\abs{x}/d$. We have
\[\begin{split}
\DeltaB^{\alpha}(\rho\abs{x}^2)
&=\nabla\cdot\left(d^\alpha\rho^{N-2}(\rho^2\abs{x}^2x+2\rho2x)\right)\\
&=\nabla\cdot\left(d^{\alpha}\rho^{N}(\abs{x}^2+(1-\abs{x}^2))x\right)\\
&=\nabla\cdot\left(d^{\alpha}\rho^{N}x\right)\\
&=d^{\alpha}\rho^{N}\left(N+N\rho\abs{x}^2+\alpha\rho\dfrac{\abs{x}}{d}\right)\\
\end{split}\]
On the other hand,
\[\begin{split}
\DeltaB^{\alpha}\left(\dfrac{\rho\abs{x}}{d}\right)
&=\nabla\cdot\left(d^{\alpha}\rho^{N-2}\left(\dfrac{\rho}{d}\dfrac{x}{\abs{x}}+\dfrac{\rho^2\abs{x}x}{d}-\dfrac{\rho^2x}{d^2}\right)\right)\\
&=\nabla\cdot\left(d^{\alpha-1}\rho^{N}x\left(\dfrac{1}{\rho\abs{x}}+\abs{x}-\dfrac{1}{d}\right)\right)\\
&=\nabla\cdot\left(d^{\alpha-1}\rho^{N}x\right)\left(\dfrac{1}{\rho\abs{x}}+\abs{x}-\dfrac{1}{d}\right)+d^{\alpha-1}\rho^{N}x\cdot\left(-\dfrac{x}{\abs{x}}-\dfrac{x}{\rho\abs{x}^3}+\dfrac{x}{\abs{x}}+\dfrac{\rho}{d^2}\dfrac{x}{\abs{x}}\right)\\
&=d^{\alpha-1}\rho^{N}\left(N+N\rho\abs{x}^2+(\alpha-1)\rho\dfrac{\abs{x}}{d}\right)\left(\dfrac{1}{\rho\abs{x}}+\abs{x}-\dfrac{1}{d}\right)
    +d^{\alpha-1}\rho^{N}\left(\dfrac{\rho\abs{x}}{d^2}-\dfrac{1}{\rho\abs{x}}\right)
\end{split}\]
With $A,B$ as in Lemma \ref{lem:Pohozaevterms}, we have
\[\begin{split}
\DeltaB^{\alpha}\left(\dfrac{\nabla\cdot(d^{\beta}\rho^{N}x)}{d^{\beta}\rho^{N}}\right)
&=N\DeltaB^{\alpha}(\rho\abs{x}^2)+\beta\DeltaB^{\alpha}\left(\dfrac{\rho\abs{x}}{d}\right)\\
\dfrac{1}{d^{\alpha-1}\rho^{N}}\DeltaB^{\alpha}\left(\dfrac{\nabla\cdot(d^{\beta}\rho^{N}x)}{d^{\beta}\rho^{N}}\right)
&=Nd\left(NA+(N+\alpha)B\right)+\dfrac{\beta}{\rho\abs{x}}\left((NA+(N-1+\alpha)B)A+B^2-1\right)\\
&=(NA+(N-1+\alpha)B)\left(Nd+\dfrac{\beta}{\rho\abs{x}}A\right)+NdB+\dfrac{\beta}{\rho\abs{x}}(B^2-1).
\end{split}\]
Since $N-1+\alpha>1>0$, $A\geq0$ and $B^2-1\geq0$, we may use $\beta>-N$ to finally yield
\[\begin{split}
\dfrac{\rho\abs{x}}{Nd^{\alpha-1}\rho^{N}}\DeltaB^{\alpha}\left(\dfrac{\nabla\cdot(d^{\beta}\rho^{N}x)}{d^{\beta}\rho^{N}}\right)
&\geq(NA+(N-1+\alpha)B)(d\rho\abs{x}-A)+d\rho\abs{x}B-(B^2-1)\\
&\geq0,
\end{split}\]
in view of (1).
\end{enumerate}
This completes the proof.
\end{proof}

\noindent


\begin{thebibliography}{99}
\bibitem{Abdellaoui-Colorado-Peral} B. Abdellaoui, E. Colorado and I. Peral, Some improved Caffarelli-Kohn-Nirenberg inequalities. {\it Calc. Var.} {\bf 23} (2005), 327--345.
\bibitem{AR} A. Ambrosetti and P. H. Rabinowitz,  Dual variational methods in critical point theory and applications, {\it J. Functional Analysis}  {\bf  14} (1973)  349--381.
\bibitem{antonini} P. Antonini, D. Mugnai and P. Pucci, Quasilinear elliptic inequalities on complete riemannian manifolds, {\it J. Math. Pures Appl.}  {\bf 87} (2007) 582--600.
\bibitem{Aubin} T. Aubin, Probl\`{e}mes isop\'{e}rim\'{e}triques de Sobolev, {\it J. Differential Geom.} {\bf11} (1976), 573--598.
\bibitem{Benguria-Benguria} R. D. Benguria and S. Benguria, An improved bound for the non-existence of radial solutions of the Brezis-Nirenberg problem in $\mathbb{H}^n$. {\it Functional analysis and operator theory for quantum physics}, 153--160, EMS Ser. Congr. Rep., Eur. Math. Soc., Z\"{u}rich, 2017.
\bibitem{bianchi} G. Bianchi, J. Chabrowski and A. Szulkin, On symmetric solutions of an elliptic equation with nonlinearity involving critical Sobolev exponent, {\it Nonlinear Anal.}  {\bf  25} (1995) 41--59.
\bibitem{bhaktasandeep} M. Bhakta and K. Sandeep, Poincar\'e-Sobolev equations in the hyperbolic spaces, {\it Calc. Var. Partial Differential Equations}  {\bf 44} (2012)  247--269.

\bibitem{CKN} L. Caffarelli, R. Kohn and L. Nirenberg, First order interpolation inequalities with weights, {\it Compositio Math.}  {\bf  53}  (1984) 259--275.
\bibitem{CFM} P. C. Carri\~{a}o, L. F. O. Faria and  O. H. Miyagaki, Semilinear elliptic equations of the H\'{e}non-type in hyperbolic space,  {\it Commun. Contemp. Math.}  {\bf 18}  (2016) [13 pages].

\bibitem{CW} F. Catarina, Z.-Q. Wang, On the Caffarelli--Kohn--Nirenberg inequalities: sharp constants, existence (and nonexistence), and symmetry of extremal functions. {\it Comm. Pure Appl. Math.} {\bf54} (2001), no. 2, 229--258.
\bibitem{Chou-Chu} K. S. Chou and C. W. Chu, On the best constant for a weighted Sobolev--Hardy inequality, {\it J. London Math. Soc.} {\bf 2} (1993), 137--151.
\bibitem{co} C. Cowan, Optimal Hardy inequalities for general elliptic operators with improvements, {\it  Commun. Pure Appl. Anal.} {\bf 9} (2010),  109-140.
\bibitem{DELG} J. Dolbeault, M. J. Esteban, M. Loss, and G. Tarantello, On the symmetry of extremals for the Caffarelli-Kohn-Nirenberg inequalities, Adv. Nonlinear Stud. 9 (2009), no. 4, 713–726.
\bibitem{Felli} V. Felli, Elliptic variational problems with critical exponent, {\it Ph.D. thesis, Scuola Internazionale Superiore di Studi Avanzati}, 2003.
\bibitem{Ganguly-Sandeep} D. Ganguly and K. Sandeep, Sign changing solutions of the Brezis-Nirenberg problem in the hyperbolic space. {\it Calc. Var. Partial Differential Equations} {\bf 50} (2014), no. 1-2, 69--91.
\bibitem{Gidas-Spruck} B. Gidas and J. Spruck, A priori bounds for positive solutions of nonlinear elliptic equations, {\it Comm. Partial Differential Equations} {\bf 6} (1981), no. 8, 883--901.
\bibitem{GNN}  B. Gidas, W. M. Ni and L. Nirenberg, Symmetry of positive solutions of nonlinear elliptic equations in $\mathbb{R}^N$, {\it Mathematical Analysis and Applications A, Adv. in Math. Suppl. Stud.}  {\bf 7a} 369--402 Academic Press (1981).
\bibitem{He} H. He, The existence of solutions for H\'{e}non equation in hyperbolic space, {\it Proc. Japan Acad.} {\bf 89}, Ser. A (2013), 24--28.
\bibitem{Hebey}  E. Hebey, Sobolev spaces on Riemannian manifolds. {\it Lecture Notes in Mathematics, 1635}. Springer-Verlag, Berlin, 1996. x+116 pp.
\bibitem{Henon}  M. H\'{e}non, Numerical experiments on the stability of spherical stellar systems, {\it Astron. Astrophys.}  {\bf 24}  (1973) 229--238.
\bibitem{ko} I. Kombe and M. \"{O}zaydin, Hardy-Poincar\'{e}, Rellich and uncertainty principle inequalities on Riemannian manifolds, {\it Trans. Amer. Math. Soc.} {\bf 365}  (2013) 5035--5050.
\bibitem{Lieb} E. H. Lieb, Sharp constants in the Hardy--Littlewood--Sobolev and related inequalities, {\it Ann. Math.} {\bf 118} (1983), 349--374.
\bibitem{mancinisandeep} G. Mancini and K. Sandeep, On a semilinear elliptic equation in $\mathbb{H}^n$, {\it  Ann. Sc. Norm. Super. Pisa Cl. Sci.}  {\bf 7} (2008) 635--671.
\bibitem{Napoli-Drelichman-Duran} P. L. De N\'{a}poli, I. Drelichman and R. G. Dur\'{a}n, Improved Caffarelli-Kohn-Nirenberg and trace inequalities for radial functions. {\it Commun. Pure Appl. Anal.} {\bf 11} (2012), no. 5, 1629--1642.
\bibitem{ni} W. M. Ni, A nonlinear Dirichlet problem on the unit ball and its applications, {\it Indiana Univ. Math. J.}  {\bf 31}  (1982) 801--807.
\bibitem{Pohozaev} S. Poho\v{z}aev, Eigenfunctions of the equation $\Delta{u}+\lambda{f}(u)=0$. {\it Soviet Math. Doklady} {\bf6} (1965), 1408--1411.
\bibitem{raticlife}  J. G. Raticliffe, Foundations of hyperbolic manifolds, {\it Graduate Texts in Mathematics}, Vol.~149 (Springer-Verlag, New York, 1994).
\bibitem{Sano-Takahashi} M. Sano and F. Takahashi, Some improvements for a class of the Caffarelli-Kohn-Nirenberg inequalities. {\it Differential Integral Equations} {\bf 31} (2018), no. 1-2, 57--74.
\bibitem{Shen-Chen} Y. Shen and Z. Chen, Some improved Caffarelli-Kohn-Nirenberg inequalities with general weights and optimal remainders. {\it Pure Appl. Math. Q.} {\bf6} (2010), no. 4, Special Issue: In honor of Joseph J. Kohn. Part 2, 1123--1143.
\bibitem{Strauss} W. A. Strauss, Existence of solitary waves in higher dimensions. {\it Comm. Math. Phys.} {\bf 55} (1977), no. 2, 149--162.
\bibitem{Talenti} G. Talenti, Best constant in Sobolev inequality, {\it Ann. Mat. Pura Appl.} {\bf 110} (1976), 353--372.
\bibitem{willem}  M. Willem,  Minimax Theorems, {\it Progress in Nonlinear Differential Equations and their Applications}, Vol.~ 24 (Birkh$\ddot{a}$user Boston, Inc., Boston, 1996).


\end{thebibliography}
\end{document}